\documentclass[10pt,reqno]{amsart}
\usepackage{bbm}
\usepackage{mathrsfs}
\usepackage{amsfonts} 
\usepackage[dvipsnames,usenames]{color}
\usepackage{graphicx,latexsym,bm,amsmath,amssymb,verbatim,multicol,lscape}
\newtheorem{thm}{Theorem} [section]

\newtheorem{lem}[thm]{Lemma}
\newtheorem{exm}[thm]{Example}

\theoremstyle{definition}
\newtheorem{defn}[thm]{Definition}
\theoremstyle{remark}

\numberwithin{equation}{section}

\begin{document}
\title[Igusa local zeta functions of a class of hybrid polynomials]
{Igusa local zeta functions of a class of hybrid polynomials}%
\author[Q.Y. Yin]{Qiuyu Yin}
\address{Mathematical College, Sichuan University, Chengdu 610064, P.R. China}
\email{yinqiuyu26@126.com}
\author[S.F. Hong]{Shaofang Hong*}
\address{Mathematical College, Sichuan University, Chengdu 610064, P.R. China}
\email{sfhong@scu.edu.cn, s-f.hong@tom.com, hongsf02@yahoo.com }
\thanks{*Hong is the corresponding author and was supported partially
by National Science Foundation of China Grant \#11371260.}

\date{\today}%
\begin{abstract}
In this paper, we study the Igusa's local zeta functions of a
class of hybrid polynomials with coefficients in a non-archimedean
local field of positive characteristic. Such class of hybrid
polynomial was first introduced by Hauser in 2003 to study
the resolution of singularities in positive characteristic.
We prove the rationality of these local zeta functions and
describe explicitly their poles. The proof is based on
Igusa's stationary phase formula.
\end{abstract}

\maketitle
\section{\bf Introduction and the main results}
Let $K$ be a non-archimedean local field with $\mathcal{O}_K$ as
its ring of integers. Let $\mathcal{P}_K$ denote the
maximal ideal of $\mathcal{O}_K$ and $\pi$ be a fixed uniformizing
parameter. Let the residue field of $K$ be $\mathbb{F}_q$, the finite
field of $q=p^r$ elements with $p$ being a prime number. For $x\in K$, we denote by
${\rm ord}(x)$ the {\it valuation} of $K$ such that ${\rm ord}(\pi)=1$,
and $|x|_K:=|x|=q^{-{\rm ord}(x)}$ is its {\it absolute value}.
For $x\in K$, the {\it angular component} of $x$, denoted by $ac(x)$,
is defined as $ac(x):=x\pi^{-{\rm ord}(x)}$.
Let $f(x)\in \mathcal{O}_K[x]$ with $x=(x_1, \cdots, x_n)$ be a
non-constant polynomial, and $\chi: \mathcal{O}_K^{\times}\rightarrow \mathbb{C}^{\times}$
be any given multiplicative character of $\mathcal{O}_K^{\times}$,
the group of units of $\mathcal{O}_K$. We put $\chi(0):=0$. Then
for any $s\in\mathbb{C}, {\rm Re}(s)>0$, the {\it Igusa's local
zeta function}, denoted by $Z(f, s, \chi, K)$, or written
$Z_f(s, \chi)$ for brevity, is defined as
$$Z(f, s, \chi, K)=Z_f(s, \chi):=\int_{\mathcal{O}_K^{n}}{\chi(ac f(x))|f(x)|^s|dx|},$$
where $|dx|$ is the Haar measure on $K^n$, normalized such that the
measure of $\mathcal{O}_K^{n}$ is one.

In the case ${\rm char}(K)=0$, Igusa \cite{[Igu1]} \cite{[Igu2]} proved that
$Z_f(s, \chi)$ is a rational function of $q^{-s}$, so it can be extended to
a meromorphic function on $\mathbb{C}$. In his proof, Igusa used the resolution
of singularities. If $\chi=\chi_{{\rm triv}}$, Denef \cite{[De1]} gave a proof
of the rationality of $Z_f(s, \chi_{{\rm triv}})$ without using resolution of
singularities. Furthermore, $Z_f(s, \chi_{{\rm triv}})$ is related to the number
of solutions of congruence $f(x)\equiv 0 \pmod{\pi^i\mathcal{O}_K}$
with $i\in \mathbb{Z}^{+}$. More precisely, we denote
$$N_i:=\#\{x\in (\mathcal{O}_K/\pi^i\mathcal{O}_K)^{n}
|f(x)\equiv 0 \pmod{\pi^i\mathcal{O}_K}\},$$
and set $N_0=1$. Let $P(t)$ be the {\it Poincar$\acute{e}$ series}
of $f$, that is,
$$P(t):=\sum_{i=0}^{\infty}{\dfrac{N_it^i}{q^{-ni}}}.$$
Then we have (see, for example, \cite{[De2]} and \cite{[Igu4]})
$$P(q^{-s})=\dfrac{1-q^{-s}Z_f(s, \chi_{{\rm triv}})}{1-q^{-s}}.$$
Therefore, if we can calculate the explicit form of $Z_f(s, \chi_{{\rm triv}})$,
then we obtain the formulas for all $N_i$. For an arbitrary character $\chi$,
it's known that $Z_f(s, \chi)$ is related to the exponential sums, see
\cite{[De2]} for more details.

However, much less is known if ${\rm char}(K)>0$, since the
techniques used by Igusa and Denef are not available in this case.
In particular, for any $f(x)\in \mathcal{O}_K[x]$,
the rationality of $Z_f(s, \chi)$ is still an open problem.
Z$\acute{u}$$\tilde{n}$iga-Galindo \cite{[ZG1]} proved that
$Z_f(s, \chi_{{\rm triv}})$ is a rational function of $q^{-s}$
when $f$ is a semiquasihomogeneous polynomial with coefficients
in an arbitrarily non-archimedean local field.
He also showed in \cite{[ZG2]} that if $f$ is a polynomial
globally non-degenerate with respect to its Newton polyhedra, then
$Z_f(s, \chi)$ is a rational function of $q^{-s}$. The basic tool he used
is called the {\it $\pi$-adic stationary phase formula}, which was first
introduced by Igusa \cite{[Igu3]} and then became a very powerful tool
in the study of the rationality of $Z_f(s, \chi)$ in arbitrary
characteristic case. For more progress in the rationality
of Igusa's local zeta function, see \cite{[M]}.

In this paper, we study the local zeta functions of a class of hybrid
polynomials $f$ in three variables, and we will prove that for these $f$,
$Z_f(s, \chi)$ is a rational function of $q^{-s}$. Moreover, we will
give a list of candidate poles of $Z_f(s, \chi)$.
From now on, let $K$ be a non-archimedean local field of characteristic
$p$, where $p$ is a fixed prime number. First of all,
we give the definition of hybrid polynomials in three variables.
\begin{defn}\label{defn 1.1}
Let $k$, $r$ and $l$ be positive integers with
$p\nmid rl$, $p\mid (r+l+k)$ and $\bar{r}+\bar{l}\leq p$, where $\bar{r}$ and
$\bar{l}$ denote the residues of $r$ and $l$ modulo $p$. If $t$ is an arbitrary
constant from the ground field $K$, then the polynomial
\begin{equation}\label{eqno 1.1}
f(x, y, z)=x^p+y^rz^l\sum_{i=0}^{k}{\dbinom{k+r}{i+r}y^i(tz-y)^{k-i}}
\end{equation}
is called {\it hybrid polynomial}.
\end{defn}

Hybrid polynomials were first introduced by Hauser (see \cite{[Ha1]}
or \cite{[Ha2]}) to study the resolution of singularities in positive
characteristic. When $r=l=1$ and $t^{k+1}$ is a $p$-th power of an
arbitrary unit from $\mathcal{O}_K$, Le$\acute{o}$n-Cardenal, Ibadula
and Segers \cite{[LIS]} used the method of \cite{[ZG1]} to prove that
$Z_f(s, \chi)$ is a rational function of $q^{-s}$. More explicitly,
they provided a list with just three candidate poles of $Z_f(s, \chi)$.

In the present paper, we consider the local zeta functions for a more general
class of hybrid polynomials, that is, we study the local zeta functions for
the following hybrid polynomials:
\begin{equation}\label{eqno 1.2}
g(x, y, z)=x^p+yz^l\sum_{i=0}^{k}{\dbinom{k+1}{i+1}y^i(tz-y)^{k-i}},
\end{equation}
where $l>1$ is an integer and satisfies the condition in
Definition \ref{defn 1.1}, and $t^{k+1}$ is a $p$-th power of
an arbitrary unit from $\mathcal{O}_K$. For these $f$,
we shall prove the rationality of $Z_f(s, \chi)$
and give the candidate poles of $Z_f(s, \chi)$.

Now let $g(x, y, z)$ be the hybrid polynomial of the
form (\ref{eqno 1.2}). First, we make change of variables
of the form: $(x, y, z)\mapsto (x_1, tz_1+y_1, z_1)$,
then we get that

\begin{align}\label{1.3}
g_1(x_1, y_1, z_1)&=x^p_1+(tz_1+y_1)z^l_1\sum_{i=0}^{k}
{\dbinom{k+1}{i+1}(tz_1+y_1)^i(-y_1)^{k-i}}\nonumber\\
&=x^p_1+z^l_1\sum_{i=0}^{k}{\dbinom{k+1}{i+1}
(tz_1+y_1)^{i+1}(-y_1)^{k-i}}\nonumber\\
&=x^p_1+z^l_1\sum_{i=1}^{k+1}{\dbinom{k+1}{i}
(tz_1+y_1)^{i}(-y_1)^{k+1-i}}\nonumber\\
&=x^p_1+t^{k+1}z^{k+1+l}_1+(-1)^{k}y^{k+1}_1z^l_1.
\end{align}
The above discussion tells us that we can consider the local
zeta functions of polynomials of following general form:
\begin{equation}\label{eqno 1.4}
f(x, y, z)=x^p+\alpha y^nz^l+\beta z^{n+l},
\end{equation}
where $\alpha\in \mathcal{O}_K^{\times}$, $\beta$ is a $p$-th power of
an arbitrary unit from $\mathcal{O}_K$ and $n$ is a positive integer
with $p\mid (n+l)$. If $l=1$, then the local zeta function $Z_f(s, \chi)$
was studied in \cite{[LIS]}. In this paper, we are able to treat with the
general case $l>1$. Actually, we have the following main results of this paper.
\begin{thm}\label{thm 1.1}
Let $f(x, y, z)=x^p+\alpha y^nz^l+\beta z^{n+l}$,
where $\alpha\in \mathcal{O}_K^{\times}$, $\beta^{1/p}\in\mathcal{O}_K^{\times}$.
$n$ is a positive integer and $l>1$ is an integer
such that $p\nmid l$ and $p\mid (n+l)$.
Then the local zeta function $Z_f(s, \chi)$
is a rational function of $q^{-s}$, and furthermore, we have
$$Z_f(s, \chi)=\dfrac{L(q^{-s}, \chi)}{(1-q^{-1-s})(1-q^{-((n+l)/p+2)-(n+l)s})
(1-q^{-p-n-nps})(1-q^{-p-l-pls})},$$
where $L(q^{-s}, \chi)$ is a polynomial with complex coefficients.
\end{thm}

Notice that if $g$ is of the form (\ref{eqno 1.2}) and $g^{'}$
is of the form (\ref{1.3}), then one can easily show that
$Z_g(s, \chi)=Z_{g^{'}}(s, \chi)$. Thus Theorem \ref{thm 1.1}
applied to $g'$ gives immediately the following result.

\begin{thm}\label{thm 1.2}
Let $g(x, y, z)$ be the hybrid polynomial of the form (\ref{eqno 1.2}).
Then the local zeta function $Z_g(s, \chi)$ is a rational function
of $q^{-s}$. Moreover, we have
$$Z_g(s, \chi)=\dfrac{L(q^{-s}, \chi)}{(1-q^{-1-s})(1-q^{-((k+l+1)/p+2)-(k+l+1)s})
(1-q^{-p-(k+1)-(k+1)ps})(1-q^{-p-l-pls})},$$
where $L(q^{-s}, \chi)$ is a polynomial with complex coefficients.
\end{thm}

Using Theorem \ref{thm 1.2} and some standard results in complex analysis,
we can provide a list of candidate poles of local zeta function of the
hybrid polynomial of the form (\ref{eqno 1.2}).

\begin{thm}\label{thm 1.3}
Let $g(x, y, z)$ be the hybrid polynomial of the form (\ref{eqno 1.2}).
If $s$ is a pole of $Z_g(s, \chi)$, then $s$ equals one of the following
four forms:
\begin{align}
&-1+\dfrac{2\pi i\mathbb{Z}}{\log q},\ \ \
-\dfrac{1}{p}-\dfrac{2}{k+1+l}+\dfrac{2\pi i\mathbb{Z}}{(k+1+l)\log q}, \nonumber\\
&-\dfrac{1}{p}-\dfrac{1}{k+1}+\dfrac{2\pi i\mathbb{Z}}{p(k+1)\log q},
\ \ \ -\dfrac{1}{p}-\dfrac{1}{l}+\dfrac{2\pi i\mathbb{Z}}{pl\log q}. \nonumber
\end{align}
\end{thm}

\noindent{\bf Remark 1.5.}
By Theorem \ref{thm 1.2}, we obtain the same three poles
as the case considered in \cite{[LIS]}, but we also get a new extra pole,
which depends on the characteristic of $K$ and the parameter $l$.

The paper is organized as follows. In Section 2, we review the
stationary phase formula and some results about Newton polyhedra.
Then we state some lemmas which are needed in the proof of
our main theorems. Finally, in Section 3, we give the proof
of Theorem \ref{thm 1.1} and supply also an example to illustrate
the validity of Theorem \ref{thm 1.1}.

\section{\bf Preliminaries and lemmas }

\subsection{\bf Stationary phase formula}

In \cite{[Igu3]}, Igusa introduced the stationary phase formula for
$\pi$-adic integrals. He suggested that a detailed study of this formula
might lead to a proof of the rationality of the local zeta function in
arbitrary characteristic. This formula is our main tool in this paper.
In this section, we review the stationary phase formula.

We denote by $\bar{x}$ the image of an element $x\in\mathcal{O}_K^n$ under
the canonical homomorphism
$\mathcal{O}_K^n\rightarrow (\mathcal{O}_K/\pi\mathcal{O}_K)^n\cong \mathbb{F}_q^n$.
For $f(x)\in \mathcal{O}_K[x]$, $\bar{f} (x)$ stands for the polynomial
obtained by reducing modulo $\pi$ the coefficients of $f(x)$.
Let $A$ be any ring and $f(x)\in A[x]$. We denote by $V_f(A)$
the set of $A$-value points of the hypersurface $V_f$ defined by $f$, i.e.
$V_f(A):=\{x\in A^n|f(x)=0\}$. By Sing$_f(A)$ we denote the set
of $A$-value singular points of $V_f$, namely,
$${\rm Sing}_f(A):=\Big\{x\in A^n\Big|f(x)
=\dfrac{\partial f}{\partial x_1}(x)=\cdots
=\dfrac{\partial f}{\partial x_n}(x)=0\Big\}.$$

We fix a lifting $R$ of $\mathbb{F}_q$ in $\mathcal{O}_K$. That is,
the set $R^n$ is mapped bijectively onto $\mathbb{F}_q^n$ by the
canonical homomorphism
$\mathcal{O}_K^n\rightarrow (\mathcal{O}_K/\pi\mathcal{O}_K)^n$.
Let $\bar{D}$ be a subset of $\mathbb{F}_q^n$ and $D$ be its preimage
under the canonical homomorphism. We also denote by $S(f, D)$ the subset
of $R^n$ mapped bijectively to the set
${\rm Sing}_{\bar{f}}(\mathbb{F}_q)\cap \bar{D}$. Furthermore, we denote
\begin{align}
v(\bar{f}, D, \chi):={\left\{\begin{array}{rl}
q^{-n} \cdot \#\{\bar{P}\in \bar{D}|\bar{P}\notin V_{\bar{f}}(\mathbb{F}_q)\}, \ \ \ \
&{\rm if} \ \chi=\chi_{{\rm triv}}, \nonumber\\
q^{-nc_{\chi}}\sum_{\{P\in D|\bar{P}\notin V_{\bar{f}}(\mathbb{F}_q)\}
{\rm mod}\ \mathcal{P}_K^{c_{\chi}}}{\chi(ac f(P)),}\  \ \ \
&{\rm if} \ \chi \neq \chi_{{\rm triv}},\nonumber
\end{array}\right.}
\end{align}
where $c_{\chi}$ is the conductor of $\chi$, and
\begin{align}
\sigma(\bar{f}, D, \chi):={\left\{\begin{array}{rl}
q^{-n}\cdot \#\{\bar{P}\in \bar{D}|\bar{P}\ {\rm is\ a\ nonsingular\ point\ of} \
V_{\bar{f}}(\mathbb{F}_q)\}, \ \  &{\rm if} \ \chi=\chi_{{\rm triv},}\nonumber\\
0,\ \   & {\rm if} \ \chi \neq \chi_{{\rm triv}}.\nonumber
\end{array}\right.}
\end{align}
Finally, we set
$$Z_f(s, \chi, D):=\int_D \chi(ac(f(x)))|f(x)|^s|dx|.$$

Now we can state Igusa's stationary phase formula as follows.

\begin{lem}({\rm Stationary phase formula})
(\cite{[Igu3]} \cite{[ZG2]})\label{lem 2.1}
For any complex number $s$ with ${\rm Re}(s)>0$, we have
\begin{align}
Z_f(s, \chi, D)=v(\bar{f}, D, \chi)&+\sigma(\bar{f}, D, \chi)
\dfrac{(1-q^{-1})q^{-s}}{1-q^{-1-s}}\nonumber\\
&+\sum_{P\in S(f, D)}{q^{-n}
\int_{\mathcal{O}_K^n} \chi(ac(f_P(x)))|f_P(x)|^s|dx|}\nonumber,
\end{align}
where $f_P(x):=f(P+\pi x)$.
\end{lem}
For the convenience of calculation, we need the following Lemma.
The lemma is also called the stationary phase formula, and we
give its proof here.

\begin{lem}\label{lem 2.2}
For any complex number $s$ with ${\rm Re}(s)>0$, we have
\begin{align*}
Z_f(s, \chi, D)=v(\bar{f}, D, \chi)&+\sigma(\bar{f}, D, \chi)
\dfrac{(1-q^{-1})q^{-s}}{1-q^{-1-s}}+Z_f(s, \chi, D_{S(f, D)}),
\end{align*}
where $D_{S(f, D)}:=\bigcup_{P\in S(f, D)}D_P$ with
$D_P:=\{x\in\mathcal{O}_K^n|x-P\in(\pi\mathcal{O}_K)^n\}$,
that is, $D_{S(f, D)}$ is the preimage of
${\rm Sing}_{\bar{f}}(\mathbb{F}_q)\cap \bar{D}$
under the canonical homomorphism
$\mathcal{O}_K^n\rightarrow (\mathcal{O}_K/\pi\mathcal{O}_K)^n$.
\end{lem}

\begin{proof}
By Lemma \ref{lem 2.1}, we only need to show that
\begin{equation}\label{eqno 2,1}
\sum_{P\in S(f, D)}{q^{-n}\int_{\mathcal{O}_K^n} \chi(ac(f_P(x)))|f_P(x)|^s|dx|
=Z_f(s, \chi, D_{S(f, D)})}.
\end{equation}
To prove that (\ref{eqno 2,1}) holds, for any
$P=(p_1,\cdots, p_n)\in S(f, D)$, we let
\begin{align*}
I_P:=&\int_{\mathcal{O}_K^n} \chi(ac(f_P(x)))|f_P(x)|^s|dx|\\
=&\int_{\mathcal{O}_K^n} \chi(ac(f(\pi x_1+p_1,\cdots, \pi x_n+p_n)))
|f(\pi x_1+p_1,\cdots, \pi x_n+p_n)|^s|dx_1\cdots dx_n|.
\end{align*}
We change variables of form:
$(x_1,\cdots, x_n)\mapsto\frac{1}{\pi}(y_1-p_1,\cdots, y_n-p_n)$.
Then
\begin{align*}
I_P=&\int_{D_P} \chi(ac(f(y_1,\cdots, y_n)))
|f(y_1,\cdots, y_n)|^s|\pi^{-n}||dy_1\cdots dy_n|\\
=&q^n\int_{D_P} \chi(ac(f(y_1,\cdots, y_n))
|f(y_1,\cdots, y_n)|^s|dy_1\cdots dy_n|\\
=&q^nZ_f(s, \chi, D_P).
\end{align*}
Moreover, by the definition of $D_P$, we can derive that
$D_P\bigcap D_P^{'}=\emptyset$ for $P, P^{'}\in S(f, D)$ and
$P\ne P^{'}$. Thus
\begin{align*}
\sum_{P\in S(f, D)}{q^{-n}I_P}
=&\sum_{P\in S(f, D)}{Z_f(s, \chi, D_P)}\\
=&Z_f(s, \chi, \bigcup_{P\in S(f, D)}D_P)\\
=&Z_f(s, \chi, D_{S(f, D)})
\end{align*}
as desired. This completes the proof of Lemma \ref{lem 2.2}.
\end{proof}

\subsection{\bf Newton polyhedra}
In this section, we review some results about Newton polyhedra.
One can see more details in \cite{[De3]} and \cite{[ZG2]}.

We set $\mathbb{R}_{+}:=\{x\in \mathbb{R}|x\geq 0\}$. Let
$f(x)=\sum_{l}{a_lx^l}\in K[x]$ be a polynomial in $n$ variables satisfying
$f(0)=0$, where the notation
$a_lx^l=a_{l_1, \cdots l_n}{x_1}^{l_1}\cdots {x_n}^{l_n}$, $l=(l_1, \cdots, l_n)$.
The set ${\rm supp}(f)=\{l\in \mathbb{N}^n|a_l\neq 0\}$ is called the
{\it support set} of $f$. Then we define the Newton polyhedra $\Gamma(f)$
of $f$ as the convex hull in $\mathbb{R}_{+}^n$ of the set
$$\bigcup_{l\in {\rm supp}(f)}{(l+\mathbb{R}_{+}^n)}.$$

By a proper {\it face} $\gamma$ of $\Gamma(f)$, we mean the non-empty convex set
$\gamma$ obtained by intersecting $\Gamma(f)$ with an affine hyperplane $H$,
such that $\Gamma(f)$ is contained in one of two half-plane determined by $H$.
The hyperplane $H$ is called the {\it supporting hyperplane} of $\gamma$. Let
$a_{\gamma}=(a_1, a_2,\cdots, a_n)\in \mathbb{N}^n\backslash\{0\}$
denote the vector that is perpendicular with the supporting
hyperplane $H$ and let  $|a_{\gamma}|:=\sum_{i}{a_i}$.
A face of codimension one is named {\it facet}.
Let $\langle , \rangle$ denote the usual inner product of $\mathbb{R}^n$,
and identify $\mathbb{R}^n$ with its dual by means of it.
For any $a\in \mathbb{R}^n$, we define
$$m(a):=\inf_{b\in \Gamma(f)}{\{\langle a, b\rangle \}}$$
and we set
$$\langle a_{\gamma}, x\rangle =m(a_{\gamma})$$
to be the equation of the supporting hyperplane $H$ of
a facet $\gamma$ with perpendicular vector $a_{\gamma}$.
We introduce the following concept (see, for example, {\rm Definition 1.1} of \cite{[ZG2]}).

\begin{defn} \label{defn 2.3}
A polynomial $f(x)=\sum_{i}{a_i}x^i\in K[x]$ is called {\it globally non-degenerate
with respect to its Newton polyhedra $\Gamma(f)$}, if it satisfies the following
two properties:

(GND 1) The origin of $K^n$ is a singular point of $f(x)$.

(GND 2) For every face $\gamma\subset\Gamma(f)$ (including $\Gamma(f)$ itself),
the polynomial
$$f_{\gamma}(x):=\sum_{i\in \gamma}{a_i}x^i$$
has the property that there is no $x\in (\mathbb{N}\backslash\{0\})^n$ such that
$x$ is a singular point of $f_{\gamma}$.
\end{defn}

We need also the following known result.

\begin{lem}(\cite{[ZG2]}, {\rm Theorem A})\label{lem 2.4}
Let K be a non-archimedean local field, and let $f(x)\in \mathcal{O}_K[x]$
be a polynomial globally non-degenerate with respect to its Newton polyhedra
$\Gamma(f)$. Then the Igusa's local zeta function $Z_f(s, \chi)$ is a rational
function of $q^{-s}$ satisfying that if $s$ is a pole of $Z_f(s, \chi)$, then
$$s=-\dfrac{|a_{\gamma}|}{m(a_{\gamma})}
+\dfrac{2\pi i}{\log q}\dfrac{k}{m(a_{\gamma})}, k\in \mathbb{Z}$$
for some facet $\gamma$ of $\Gamma(f)$ with perpendicular $a_{\gamma}$ if
$m(a_{\gamma})\neq 0$, and
$$s=-1+\dfrac{2\pi i}{\log q}k, k\in \mathbb{Z}$$
otherwise.
\end{lem}
Now we can show the following result.

\begin{lem}\label{lem 2.5}
Let K be a non-archimedean local field, and let
$$f(x, y)=\sum_{i=1}^{\infty}{a_ix^i}
+\sum_{j=1}^{\infty}{b_jy^j}\in \mathcal{O}_K[x, y]$$
be a two-variable polynomial globally non-degenerate with respect to its Newton polyhedra
$\Gamma(f)$. Let $i_0:=\min\{i|a_i\neq 0\}$ and $j_0:=\min\{j|b_j\neq 0\}$. Then
$$Z_f(s, \chi)=\dfrac{H(q^{-s}, \chi)}{(1-q^{-1-s})(1-q^{-i_0-j_0-i_0j_0s})},$$
where $H(q^{-s}, \chi)$ is a polynomial with complex coefficients.
\end{lem}

\begin{proof}
By the definition of $i_0$ and $j_0$, we know that $\Gamma(f)$
is the convex hull in $\mathbb{R}_{+}^2$ of the set
$$((i_0, 0)+\mathbb{R}_{+}^2)\bigcup ((0, j_0)+\mathbb{R}_{+}^2).$$

Let $\gamma$ be a facet of $\Gamma(f)$. Then
$\gamma=\{(x, y)\in\mathbb{R}^2|x=0, y\geq j_0\}$, or
$\gamma=\{(x, y)\in\mathbb{R}^2|x\geq i_0, y=0\}$,
or $\gamma=\{(x, y)\in\mathbb{R}^2
|j_0x+i_0y=i_0j_0,\ 0\leq x\leq i_0,\ 0\leq y\leq j_0\}$.

In the first two cases,  the supporting hyperplane of $\gamma$ are
$\{(x, y)\in\mathbb{R}^2|x=0\}$ and $\{(x, y)\in\mathbb{R}^2|y=0\}$.
Thus one has $a_{\gamma}=(1, 0)$ or $a_{\gamma}=(0, 1)$. So
$m(a_{\gamma})=0$ in both situations.

In the third case, one has $\gamma=\{(x, y)\in\mathbb{R}^2
|j_0x+i_0y=i_0j_0,\ 0\leq x\leq i_0,\ 0\leq y\leq j_0\}$.
Then the supporting hyperplane of $\gamma$ is
$H=\{(x, y)\in\mathbb{R}^2|j_0x+i_0y=i_0j_0\}$, with perpendicular
vector $a_{\gamma}=(j_0, i_0)$. Furthermore, one has
\begin{align*}
m(a_{\gamma})=& \inf_{(x,y)\in \Gamma(f)}{\{\langle a_{\gamma}, (x, y)\rangle \}}\\
=& \inf_{(x,y)\in \Gamma(f)}\{j_0x+i_0y\}\\
=& \inf_{(x,y)\in \Gamma(f)}\Big\{i_0j_0\Big(\frac{x}{i_0}+\frac{y}{j_0}\Big)\Big\}\\
=& i_0j_0\inf_{(x,y)\in \Gamma(f)}\Big\{\frac{x}{i_0}+\frac{y}{j_0}\Big\}\\
=& i_0j_0,
\end{align*}
the last equality is because $\frac{x}{i_0}+\frac{y}{j_0}\ge 1$
when the point $(x, y)$ runs through $\Gamma(f)$ and
$\frac{x}{i_0}+\frac{y}{j_0}=1$ if $(x, y)\in H$.

Now by using Lemma \ref{lem 2.4}, we know that the local zeta function
of $f$ is a rational function of $q^{-s}$, and its only possible poles are
$$s=-\dfrac{i_0+j_0}{i_0j_0}
+\dfrac{2\pi i}{\log q}\dfrac{k}{i_0j_0}, k\in \mathbb{Z},$$
or
$$s=-1+\dfrac{2\pi i}{\log q}k, k\in \mathbb{Z}.$$

This concludes the proof of Lemma \ref{lem 2.5}.
\end{proof}

\subsection{\bf Some lemmas}
In this section, we show some lemmas which will be used in
the proof of our main theorem.

\begin{lem}\label{lem 2.6}
Let $\chi: \mathcal{O}_K^{\times}\rightarrow \mathbb{C}^{\times}$
be a multiplicative character of $\mathcal{O}_K^{\times}$, and
$m$ be any positive integer. Then
\begin{align}
\int_{\mathcal{O}_K^{\times}}{\chi(ac (x^m))|dx|}
={\left\{\begin{array}{rl}
&1-q^{-1},\ \  \ {\rm if}\ \chi^m=\chi_{{\rm triv}},\nonumber\\
&0,\ \ \ \ \ \ \ \ \ \ \ {\rm if}\ \chi^m\neq \chi_{{\rm triv}}.\nonumber
\end{array}\right.}
\end{align}
\end{lem}

\begin{proof} We divide the proof into the following two cases:

{\sc Case 1.} $\chi^m=\chi_{{\rm triv}}$. Since $ac$ is a
multiplicative function and $\mathcal{O}_k=
\pi\mathcal{O}_K\cup\mathcal{O}_K^{\times}$, we have
\begin{align*}
\int_{\mathcal{O}_K^{\times}}{\chi(ac (x^m))|dx|}
&=\int_{\mathcal{O}_K^{\times}}{\chi^m(ac (x))|dx|}\\
&=\int_{\mathcal{O}_K^{\times}}{|dx|}\\
&=\int_{\mathcal{O}_K}{|dx|}-\int_{\pi\mathcal{O}_K}{|dx|}\\
&=\int_{\mathcal{O}_K}{|dx|}-q^{-1}\int_{\mathcal{O}_K}{|dx|}\\
&=1-q^{-1}
\end{align*}
since $|dx|$ is the Haar measure on $K$, normalized such that the
measure of $\mathcal{O}_K$ is one.

{\sc Case 2.} $\chi^m\neq \chi_{{\rm triv}}$. Then there
exists one $a\in \mathcal{O}_K^{\times}$ such that $\chi^m(a)\neq 1$.
Since $ac(a)=a\pi^{-{\rm ord}(a)}=a\pi^{-0}=a$ for
$a\in \mathcal{O}_K^{\times}$, one deduces that
\begin{align*}
\chi^m(a)\int_{\mathcal{O}_K^{\times}}{\chi(ac (x^m))|dx|}
&=\chi(ac (a^m))\int_{\mathcal{O}_K^{\times}}{\chi(ac (x^m))|dx|}\\
&=\int_{\mathcal{O}_K^{\times}}{\chi(ac ((ax)^m))|dx|}\\
&=\int_{\mathcal{O}_K^{\times}}{\chi(ac (x^m))|dx|}.
\end{align*}
But $\chi^m(a)\neq 1$. Hence
$$\int_{\mathcal{O}_K^{\times}}{\chi(ac (x^m))|dx|}=0$$
as desired. So Lemma \ref{lem 2.6} is proved.
\end{proof}
The following lemma is a generalization of Lemma 2 of \cite{[LIS]}.

\begin{lem}\label{lem 2.7}
Let $f(x, y, z)=x^p+\alpha\pi^{n}y^nz^l+\beta z^{n+l}$,
where $\alpha\in \mathcal{O}_K^{\times}$,
$\beta^{1/p}\in\mathcal{O}_K^{\times}$,
$l,\ n$ are positive integers such that $p\mid (n+l)$ and
$p\nmid l$. Then
\begin{align*}
Z_f(s, \chi, D)={\left\{\begin{array}{rl}
\dfrac{M(q^{-s}, \chi)}{(1-q^{-1-s})(1-q^{-p-n-pns})},
&{\rm if}\ \chi^{n+l}=\chi_{{\rm triv}},\\
0, &{\rm if}\ \chi^{n+l}\neq \chi_{{\rm triv}},
\end{array}\right.}
\end{align*}
where $D=\mathcal{O}_K^{\times}\times\mathcal{O}_K\times\mathcal{O}_K^{\times}$
and $M(q^{-s}, \chi)$ is a polynomial with complex coefficients.
\end{lem}

\begin{proof}
For $Z_f(s, \chi, D)$, we change variables as
$(x, y, z)\mapsto (uw^{(n+l)/p}, vw, w)$.
Since $|w|=1$ for $w\in \mathcal{O}_K^{\times}$, we have
\begin{align*}
& Z_f(s, \chi, D)\\
=&\int_D\chi(ac (x^p+\alpha\pi^ny^nz^l+\beta z^{n+l}))
|x^p+\alpha\pi^ny^nz^l+\beta z^{n+l}|^s|dxdydz|\\
=&\int_D\chi(ac(w^{n+l}(u^p+\alpha\pi^nv^n+\beta)))
|w^{n+l}(u^p+\alpha\pi^nv^n+\beta)|^s|w^{\frac{n+l}{p}+1}||dudvdw|\\
=&\int_{\mathcal{O}_K^{\times}\times\mathcal{O}_K}
\chi(ac(u^p+\alpha\pi^nv^n+\beta))
|(u^p+\alpha\pi^nv^n+\beta)|^s|dudv|
\int_{\mathcal{O}_K^{\times}}\chi (ac(w^{n+l}))|dw|\\
=& Z_g(s, \chi, D^{'}) \int_{\mathcal{O}_K^{\times}}\chi (ac(w^{n+l}))|dw|,
\end{align*}
where $g(u, v):=u^p+\alpha\pi^nv^n+\beta$ and
$D^{'}:=\mathcal{O}_K^{\times}\times\mathcal{O}_K$.
Then by Lemma \ref{lem 2.6}, we obtain that
\begin{align*}
Z_f(s, \chi, D)={\left\{\begin{array}{rl}
(1-q^{-1})Z_g(s, \chi, D^{'}),
& {\rm if} \ \ \chi^{n+l}=\chi_{{\rm triv}},\\
0, & {\rm if} \ \ \chi^{n+l}\neq \chi_{{\rm triv}}.
\end{array}\right.}
\end{align*}

Let now $\chi^{n+l}=\chi_{{\rm triv}}$.
To prove the lemma, we need to calculate $Z_g(s, \chi, D^{'})$.
Since $\beta^{1/p}\in\mathcal{O}_K^{\times}$,
one has $\bar{g}(u, v)=u^p+\bar{\beta}$ with
$\bar{\beta}^{1/p}\in\mathbb{F}_q^{\times}$.
Thus we derive that
${\rm Sing}_{\bar{g}}(\mathbb{F}_q)\bigcap \bar{D^{'}}
=\{(u, v)\in \mathbb{F}_q^2|u=-\bar{\beta}^{1/p}\}$,
which implies that
$S(g, D^{'})=\{(u,v )\in R^{2}|u=-\beta^{1/p}\}$.
Hence
$$D_{S(g, D^{'})}=\bigcup_{P\in S(g, D^{'})}D_P
=(-\beta^{1/p}+\pi\mathcal{O}_K)\times \mathcal{O}_K.$$

On the other hand, Lemma \ref{lem 2.2} tells us that
\begin{equation}\label{eqno 2.2}
Z_g(s, \chi, D^{'})=v(\bar{g}, D^{'}, \chi)
+\sigma(\bar{g}, D^{'}, \chi)\dfrac{(1-q^{-1})q^{-s}}{1-q^{-1-s}}
+Z_g(s, \chi, D_{S(g, D^{'})}).
\end{equation}
For $Z_g(s, \chi, D_{S(g, D^{'})})$, we make the change of variables
of the form $(u, v)\mapsto (-\beta^{1/p}+\pi x, y)$ to get that
\begin{align}\label{2.3}
& Z_g(s, \chi, D_{S(g, D^{'})})\nonumber\\
=&\int_{D_{S(g, D^{'})}}
\chi(ac(u^p+\alpha\pi^nv^n+\beta))
|u^p+\alpha\pi^nv^n+\beta|^s|dudv|\nonumber\\
=&\int_{\mathcal{O}_K^2}
\chi(ac((-\beta^{1/p}+\pi x)^p+\alpha\pi^ny^n+\beta))
|(-\beta^{1/p}+\pi x)^p+\alpha\pi^ny^n+\beta|^s|\pi||dxdy|\nonumber\\
=&q^{-1}\int_{\mathcal{O}_K^2}
\chi(ac(\pi^p x^p+\alpha\pi^ny^n))
|\pi^p x^p+\alpha\pi^ny^n|^s|dxdy|\nonumber\\
=& q^{-1}Z_{g_1}(s, \chi),
\end{align}
the third equality is due to the reason that $p$ is the characteristic of $K$.

Now let $g_1(u, v)=\pi^p x^p+\alpha\pi^ny^n$. It is easy to check that
$g_1$ is a polynomial globally non-degenerate with respect to its
Newton polyhedra. Using Lemma \ref{lem 2.5}, we arrive at
\begin{align}\label{2.4}
Z_{g_1}(s, \chi)=\dfrac{G(q^{-s}, \chi)}{(1-q^{-1-s})(1-q^{-p-n-pns})},
\end{align}
with $G(q^{-s}, \chi)$ being a polynomial of complex coefficients.
From (\ref{eqno 2.2}), (\ref{2.3}) and (\ref{2.4}), it then follows
immediately that
\begin{equation*}
Z_f(s, \chi, D)=(1-q^{-1})Z_g(s, \chi, D^{'})
=\dfrac{M(q^{-s}, \chi)}{(1-q^{-1-s})(1-q^{-p-n-pns})},
\end{equation*}
where $M(q^{-s}, \chi)$ is a polynomial with complex coefficients.

The proof of Lemma \ref{lem 2.7} is complete.
\end{proof}

\begin{lem}\label{lem 2.8}
Let $f(x, y, z)=\pi^{m}x^p+\alpha\pi^{l}y^nz^l+\beta\pi^{n+l}z^{n+l}$,
with $\alpha\in \mathcal{O}_K^{\times}$,
$\beta^{1/p}\in\mathcal{O}_K^{\times}$ and $m$ being a nonnegative
integer, $l$ and $n$ being the positive integers such that $p\mid (n+l)$
and $p\nmid l$. Then
\begin{align}
Z_f(s, \chi, D)
={\left\{\begin{array}{rl}
&H_1(q^{-s}, \chi),\ \ \ \ \ \  if \ 0\le m<l,\nonumber\\
&\dfrac{H_2(q^{-s}, \chi)}{1-q^{-1-s}},\ \ \ \ \  otherwise, \nonumber
\end{array}\right.}
\end{align}
where $D=(\mathcal{O}_K^{\times})^2\times\mathcal{O}_K$ and
$H_i(q^{-s}, \chi)$ are polynomials with complex coefficients.
\end{lem}

\begin{proof}
We prove this lemma by considering the following two cases.

{\sc Case 1.} $0\le m<l$. Then one has
$$ac(\pi ^m)=\pi^m\cdot \pi^{-{\rm ord}(\pi^m)}
=\pi^m\cdot \pi^{-m{\rm ord}(\pi)}=\pi^{m+(-m)}=1$$
and
$$
|\pi^m|=q^{-{\rm ord}(\pi^m)}=q^{-m{\rm ord}(\pi)}=q^{-m}.
$$
It follows that
\begin{align*}
&Z_f(s, \chi, D)\\
=&\int_D\chi(ac (\pi^mx^p+\alpha\pi^ly^nz^l+\beta\pi^{n+l}z^{n+l}))
|\pi^mx^p+\alpha\pi^ly^nz^l+\beta\pi^{n+l}z^{n+l}|^s|dxdydz|\\
=&q^{-ms}\times\\
&\int_D\chi(ac (x^p+\alpha\pi^{l-m}y^nz^l+\beta\pi^{n+l-m}z^{n+l}))
|x^p+\alpha\pi^{l-m}y^nz^l+\beta\pi^{n+l-m}z^{n+l}|^s|dxdydz|.
\end{align*}

Since $l>m$, one has $\pi|(\alpha\pi^{l-m}y^nz^l+\beta\pi^{n+l-m}z^{n+l})$.
But $x\in \mathcal{O}_K^{\times}$. Then we have
$\pi \nmid (x^p+\alpha\pi^{l-m}y^nz^l+\beta\pi^{n+l-m}z^{n+l})$
which implies that
$${\rm ord}(x^p+\alpha\pi^{l-m}y^nz^l+\beta\pi^{n+l-m}z^{n+l})=0.$$
Hence
$|x^p+\alpha\pi^{l-m}y^nz^l+\beta\pi^{n+l-m}z^{n+l}|=1$
if $(x, y, z)\in D$. Thus
\begin{equation*}
Z_f(s, \chi, D)=q^{-ms}\int_D
\chi(ac (x^p+\alpha\pi^{l-m}y^nz^l+\beta\pi^{n+l-m}z^{n+l}))|dxdydz|
:=H_1(q^{-s}, \chi).
\end{equation*}
So Lemma 2.8 is proved in this case.

{\sc Case 2.} $m\geq l$. Then we have
\begin{align*}
&Z_f(s, \chi, D)\\
=& q^{-ls}\int_D
\chi(ac (\pi^{m-l}x^p+\alpha y^nz^l+\beta\pi^nz^{n+l}))
|\pi^{m-l}x^p+\alpha y^nz^l+\beta\pi^nz^{n+l}|^s|dxdydz|.
\end{align*}

{\sc Subcase 2.1.} $m=l$. Let
$h_1(x, y, z)=x^p+\alpha y^nz^l+\beta\pi^nz^{n+l}$. Then
$\bar{h}_1=x^p+\bar{\alpha}y^nz^l$ and
$\bar{D}=(\mathbb{F}_q^{\times})^2\times \mathbb{F}_q$.
If $P=(x, y, z)\in {\rm Sing}_{\bar{h}_1}(\mathbb{F}_q)$,
then
$$\bar{h}_1(P)=\frac{\partial \bar{h}_1}{\partial x}(P)
=\frac{\partial \bar{h}_1}{\partial y}(P)
=\frac{\partial \bar{h}_1}{\partial z}(P)=\bar{0}.$$
But $p\nmid n$ since $p\nmid l$ and $p\mid (n+l)$.
Thus letting $\frac{\partial \bar{h}_1}{\partial y}(P)
=\bar{\alpha}ny^{n-1}z^l=\bar{0}$ gives us that
$zy=\bar{0}$. Then one can deduce from
$\bar{h}_1(P)=\bar{0}$ that $x=\bar{0}$. So it follows that
${\rm Sing}_{\bar{h}_1}(\mathbb{F}_q)\bigcap \bar{D}=\emptyset$,
which tells us that $S(h_1, D)=\emptyset$. Using Lemma
\ref{lem 2.1}, we have
\begin{equation}\label{eqno 2.5}
Z_f(s, \chi, D)=q^{-ls}Z_{h_1}(s, \chi, D)=
q^{-ls}\Big(c_1(\chi )+c_2(\chi )\frac{(1-q^{-1})q^{-s}}{1-q^{-1-s}}\Big)
:=\dfrac{H_{2, 1}(q^{-s}, \chi)}{1-q^{-1-s}}
\end{equation}
as one desires, where $c_1(\chi)$ and $c_2(\chi)$ are constants
depended on $\chi $ and $h_1$. So Lemma 2.8 is true in this subcase.

{\sc Subcase 2.2.} $m>l$. Let
$h_2(x, y, z)=\pi^{m-l}x^p+\alpha y^nz^l+\beta\pi^nz^{n+l}$.

If $l=1$, then $\bar{h}_2=\bar{\alpha}y^nz$. It is easy to see that
${\rm Sing}_{\bar{h}_2}(\mathbb{F}_q)\bigcap \bar{D}=\emptyset$,
which infers that $S(h_2, D)=\emptyset$.
Then by Lemma \ref{lem 2.1}, we have
\begin{equation}\label{eqno 2.6}
Z_f(s, \chi, D)=q^{-ls}Z_{h_2}(s, \chi, D)=
\dfrac{H_{2, 2}(q^{-s}, \chi)}{1-q^{-1-s}}
\end{equation}
as required, where $H_{2, 2}(q^{-s}, \chi)$ is a
polynomial with complex coefficients.

If $l>1$, then $\bar{h}_2=\bar{\alpha}y^nz^l$. It follows that
$${\rm Sing}_{\bar{h}_2}(\mathbb{F}_q)\bigcap \bar{D}
=\{(x, y, z)\in \mathbb{F}_q^3|x, y\in \mathbb{F}_q^{\times}, z=\bar{0} \}.$$
Thus $S(h_2, D)=\{(x, y, z)\in R^3|z=0, x\ne 0, y\ne 0\}$ and
$D_{S(h_2, D)}=(\mathcal{O}_K^{\times})^2\times\pi\mathcal{O}_K$.
Then using Lemma \ref{lem 2.2}, we obtain that
\begin{align}\label{2.7}
Z_f(s, \chi, D)&=q^{-ls}Z_{h_2}(s, \chi, D)\nonumber\\
&=q^{-ls}\Big(v(\bar{h}_2, D, \chi)
+\sigma(\bar{h}_2, D, \chi)\dfrac{(1-q^{-1})q^{-s}}{1-q^{-1-s}}
+Z_{h_2}(s, \chi, D_{S(h_2, D)})\Big)\nonumber\\
&=\dfrac{H_{2,3}(q^{-s}, \chi)}{1-q^{-1-s}}
+q^{-ls}Z_{h_2}(s, \chi, D_{S(h_2, D)}),
\end{align}
where $H_{2,3}(q^{-s}, \chi)$ is a polynomial with complex coefficients.

For $Z_{h_2}(s, \chi, D_{S(h_2, D)})$, we make the change of variables
of the form: $(x, y, z)\mapsto (x_1, y_1, \pi z_1)$ and get that
$|dx dy dz|=|\pi ||dx_1dy_1dz_1|=q^{-1}|dx_1dy_1dz_1|.$ Hence
\begin{align}\label{2.8}
&Z_{h_2}(s, \chi, D_{S(h_2, D)})\nonumber\\
=&\int_{D_{S(h_2, D)}}
\chi(ac (\pi^{m-l}x^p+\alpha y^nz^l+\beta\pi^nz^{n+l}))
|\pi^{m-l}x^p+\alpha y^nz^l+\beta\pi^nz^{n+l}|^s|dxdydz|\nonumber\\
=&\int_D\chi(ac (\pi^{m-l}{x_1}^p+\alpha\pi^l{y_1}^n{z_1}^l
+\beta\pi^{2n+l}{z_1}^{n+l}))\nonumber\\
& \times |\pi^{m-l}{x_1}^p+\alpha\pi^l{y_1}^n{z_1}^l
+\beta\pi^{2n+l}{z_1}^{n+l}|^sq^{-1}|dx_1dy_1dz_1|\nonumber\\
=&q^{-1}Z_{h_3}(s, \chi, D),
\end{align}
where $h_3(x, y, z):=\pi^{m-l}x^p+\alpha\pi^ly^nz^l+\beta\pi^{2n+l}z^{n+l}$.
Then putting (\ref{2.8}) into (\ref{2.7}) yields that
\begin{equation*}
Z_f(s, \chi, D)=\dfrac{H_{2,3}(q^{-s}, \chi)}{1-q^{-1-s}}
+q^{-1-ls}Z_{h_3}(s, \chi, D).
\end{equation*}

Let $m=dl+r$ with $d$ and $r$ being positive integers and
$0\leq r<l$. By Lemma \ref{lem 2.2} applied to the polynomial
$h_3$ for $d-1$ times, the above argument finally arrives at
\begin{equation}\label{eqno 2.9}
Z_{f}(s, \chi, D)=\dfrac{H_{2,4}(q^{-s}, \chi)}{1-q^{-1-s}}
+q^{-d-dls}Z_{h_4}(s, \chi, D),
\end{equation}
where $h_4(x, y, z):=\pi^rx^p+\alpha\pi^ly^nz^l+\beta\pi^{(d+1)n+l}z^{n+l}$,
and $H_{2,4}(q^{-s}, \chi)$ is a polynomial with complex coefficients.
Since $r<l$, then one can apply the result for {\sc Case 1} gives us that
\begin{equation}\label{eqno 2.10}
Z_{h_4}(s, \chi, D):=H_{2,6}(q^{-s}, \chi),
\end{equation}
with $H_{2,6}$ being a polynomial with complex coefficients.
Now from (\ref{eqno 2.9}) and (\ref{eqno 2.10}), we can derive that
\begin{equation}\label{eqno 2.11}
Z_{f}(s, \chi, D)=\dfrac{H_{2,7}(q^{-s}, \chi)}{1-q^{-1-s}}
\end{equation}
as desired, where $H_{2,7}$ is a polynomial with complex
coefficients. This finishes the proof of Lemma
\ref{lem 2.8} for Case 2. Hence Lemma \ref{lem 2.8} is proved.
\end{proof}

\begin{lem}\label{lem 2.9}
Let $x=\pi^mx_1\in \mathcal{O}_K$ with $m$ being a nonnegative integer.
Then each of the following is true.

{\rm (1).} If ${\rm ord}(x)\ge m$, then $x_1\in \mathcal{O}_K$.

{\rm (2).} If ${\rm ord}(x)=m$, then $x_1\in \mathcal{O}_K^{\times}.$
\end{lem}

\begin{proof}
If ${\rm ord}(x)\ge m$, then we have ${\rm ord}(\pi^mx_1)\ge m$. Since
${\rm ord}(\pi)=1$, one deduces that ${\rm ord}(x_1)\ge m-m=0$.
So one has $x_1\in \mathcal{O}_K$ as desired. Part (1) is proved.

If ${\rm ord}(x)=m$, then it follows that ${\rm ord}(x_1)=m-m=0$ which
infers that $x_1\in \mathcal{O}_K^{\times }$ as required.
So part (2) is true. This concludes the proof of Lemma \ref{lem 2.9}.
\end{proof}

\section{\bf Proof of Theorem \ref{thm 1.1}}

In this section, we present the proof of Theorem \ref{thm 1.1}.
First of all, for any given nonnegative integers $i, j$ and $k$,
we introduce the transform $T_{i,j,k}$ of variables defined as follows:
\begin{equation}\label{eqno 3.1}
T_{i,j,k}(x, y, z):=(\pi^ix_1, \pi^jy_1, \pi^kz_1).
\end{equation}
Let $J(T_{i,j,k})$ be the Jacobian determinant associated to $T_{i,j,k}$.
Then from the definition ($\ref{eqno 3.1}$), we deduce that $J(T_{i,j,k})=\pi^{i+j+k}$.
It follows that for any $D\subseteq \mathcal{O}_K^3$, we have
\begin{align}\label{3.2}
Z_f(s, \chi, D)
=&\int_{D}{\chi(ac (f(x, y, z)))|f(x, y, z)|^s|dxdydz|}\nonumber\\
=&\int_{D_1}
{\chi(ac (f(\pi^ix_1, \pi^jy_1, \pi^kz_1)))
|f(\pi^ix_1, \pi^jy_1, \pi^kz_1)|^s|J(T_{i,j,k})||dx_1dy_1dz_1|}\nonumber\\
=&q^{-i-j-k}\int_{D_1}
{\chi(ac (f(\pi^ix_1, \pi^jy_1, \pi^kz_1)))
|f(\pi^ix_1, \pi^jy_1, \pi^kz_1)|^s|dx_1dy_1dz_1|},
\end{align}
where $D_1$ is the domain of $(x_1, y_1, z_1)$. In what follows,
we can give the proof of Theorem \ref{thm 1.1}. \\

{\it Proof of Theorem \ref{thm 1.1}}. At first, we define a set $A$ by
$$A:=\{(x, y, z)\in\mathcal{O}_K^3|
{\rm ord}(x)\geq \omega,\ {\rm ord}(y)\geq 1,\ {\rm ord}(z)\geq 1\},$$
where $\omega=\frac{n+l}{p}$. Then by (\ref{3.2}), we have
\begin{align} \label{3.2'}
Z_f(s, \chi, A)=q^{-i-j-k}\int_{\tilde A}
{\chi(ac (f(\pi^ix, \pi^jy, \pi^kz)))
|f(\pi^ix, \pi^jy, \pi^kz)|^s|dxdydz|},
\end{align}
where
$$\tilde A=\{(x, y, z)\in K^3| (\pi^i x, \pi^j y, \pi^k z)\in\mathcal{O}_K^3,
{\rm ord}(\pi^ix)\geq \omega,\ {\rm ord}(\pi^jy)\geq 1,\ {\rm ord}(\pi^kz)\geq 1\}.$$
Now letting $i=\omega,\ j=k=1$, then Lemma \ref{lem 2.9} (1) gives us that
$\tilde A=\mathcal{O}_K^3$. So by (\ref{3.2'}), we derive that
\begin{align} \label{3.3'}
&Z_f(s, \chi, A) \nonumber \\
=&q^{-(\omega+2)}\int_{\mathcal{O}_K^3}
\chi(ac (\pi^{n+l}x^p+\alpha\pi^{n+l}y^nz^l+\beta\pi^{n+l}z^{n+l}))\nonumber\\
& \times |\pi^{n+l}x^p+\alpha\pi^{n+l}y^nz^l+\beta\pi^{n+l}z^{n+l}|^s|dxdydz|\nonumber\\
=&q^{-(\omega+2)-(n+l)s}\int_{\mathcal{O}_K^3}
\chi(ac (x^p+\alpha y^nz^l+\beta z^{n+l}))
|x^p+\alpha y^nz^l+\beta z^{n+l}|^s|dxdydz|\nonumber\\
=&q^{-(\omega+2)-(n+l)s}Z_f(s, \chi).
\end{align}
On the other hand, since $\mathcal{O}_K^3=A\cup A^c$,
where $A^c$ is the complement of $A$ in $\mathcal{O}_K^3$,
one has
$Z_f(s, \chi)=Z_f(s, \chi, A)+Z_f(s, \chi, A^c)$. Therefore
\begin{align} \label{3.3''}
Z_f(s, \chi, A)=Z_f(s, \chi)-Z_f(s, \chi, A^c).
\end{align}
So from (\ref{3.3'}) and (\ref{3.3''}) it follows that
\begin{align}\label{3.3}
Z_f(s, \chi)=\dfrac{1}{1-q^{-(\omega+2)-(n+l)s}}Z_f(s, \chi, A^c).
\end{align}

It is easy to see that $A^c$ can be
decomposed as the disjoint union of the following seven sets:
\begin{align*}
A_1&=\{(x, y, z)|{\rm ord}(x)\geq \omega,\ {\rm ord}(y)=0,\ {\rm ord}(z)\geq 1\},\\
A_2&=\{(x, y, z)|{\rm ord}(x)\geq \omega,\ {\rm ord}(y)\geq 1,\ {\rm ord}(z)=0\},\\
A_3&=\{(x, y, z)|{\rm ord}(x)\geq \omega,\ {\rm ord}(y)=0,\ {\rm ord}(z)=0\},\\
A_4&=\{(x, y, z)|0\le{\rm ord}(x)<\omega,\ {\rm ord}(y)\geq 1,\ {\rm ord}(z)\geq 1\},\\
A_5&=\{(x, y, z)|0\le{\rm ord}(x)<\omega,\ {\rm ord}(y)=0,\ {\rm ord}(z)\geq 1\},\\
A_6&=\{(x, y, z)|0\le{\rm ord}(x)<\omega,\ {\rm ord}(y)\geq 1,\ {\rm ord}(z)=0\},\\
A_7&=\{(x, y, z)|0\le{\rm ord}(x)<\omega,\ {\rm ord}(y)=0,\ {\rm ord}(z)=0\}.
\end{align*}
Then one has
\begin{equation}\label{eqno 3.4}
Z_f(s, \chi)=\dfrac{1}{1-q^{-(\omega+2)-(n+l)s}}
\sum_{i=1}^{7}Z_f(s, \chi, A_i).
\end{equation}
In order to get the formula of $Z_f(s, \chi)$, we need to calculate
the seven integrals on the right hand side of (\ref{eqno 3.4}),
which will be done in the following.

For $Z_f(s, \chi, A_1)$, we set $i=\omega,\ j=0,\ k=1$ in (\ref{eqno 3.1}).
Then by (\ref{3.2}), we have
\begin{align*}
&Z_f(s, \chi, A_1)=q^{-(\omega+1)}\\
&\times\int_{B_1}\chi(ac (\pi^{n+l}x^p+\alpha\pi^{l}y^nz^l+\beta\pi^{n+l}z^{n+l}))
|\pi^{n+l}x^p+\alpha\pi^{l}y^nz^l+\beta\pi^{n+l}z^{n+l}|^s|dxdydz|\\
=&q^{-(\omega+1)-ls}
\int_{B_1}\chi(ac (\pi^{n}x^p+\alpha y^nz^l+\beta\pi^{n}z^{n+l}))
|\pi^{n}x^p+\alpha y^nz^l+\beta\pi^{n}z^{n+l}|^s|dxdydz|,
\end{align*}
where
$$B_1=\{(x, y, z)\in K^3| (\pi^{\omega} x, y, \pi z)\in\mathcal{O}_K^3,
{\rm ord}(\pi^{\omega}x)\geq \omega,\ {\rm ord}(y)=0,\ {\rm ord}(\pi z)\geq 1\}.$$
By Lemma \ref{lem 2.9}, we get that
$B_1=\mathcal{O}_{K}\times{\mathcal{O}_{K}}^{\times}\times\mathcal{O}_{K}$.
Now we make the change of variables of the form:
$(x, y, z)\mapsto (uw^{(n+l)/p}, w, vw)$.
Since $|w|=1$ for $w\in \mathcal{O}_K^{\times}$, one has
\begin{align*}
&Z_f(s, \chi, A_1)\\
=&q^{-(\omega+1)-ls}\int_{B_1}
\chi(ac (w^{n+l}(\pi^{n}u^p+\alpha v^l+\beta\pi^{n}v^{n+l})))\\
&\times|w^{n+l}(\pi^{n}u^p+\alpha v^l+\beta\pi^{n}v^{n+l})|^s
|-w^{\frac{n+l}{p}+1}||dudvdw|\\
=&q^{-(\omega+1)-ls}\int_{\mathcal{O}_K^2}
\chi(ac (\pi^{n}u^p+\alpha v^l+\beta\pi^{n}v^{n+l}))
|\pi^{n}u^p+\alpha v^l+\beta\pi^{n}v^{n+l}|^s|dudv|\\
&\times\int_{\mathcal{O}_K^{\times}}{\chi(ac (w^{n+l}))|dw|}.
\end{align*}
By Lemma \ref{lem 2.6}, one has
\begin{align*}
\int_{\mathcal{O}_K^{\times}}{\chi(ac (x^{n+l}))|dx|}
={\left\{\begin{array}{rl}
1-q^{-1}, &{\rm if}\ \chi^{n+l}=\chi_{{\rm triv}},\\
0, &{\rm if}\ \chi^{n+l}\neq \chi_{{\rm triv}}.
\end{array}\right.}
\end{align*}
So it follows that
\begin{align*}
Z_f(s, \chi, A_1)={\left\{\begin{array}{rl}
(1-q^{-1})q^{-(\omega+1)-ls}Z_g(s, \chi),
&{\rm if}\ \chi^{n+l}=\chi_{{\rm triv}},\\
0, &{\rm if}\ \chi^{n+l}\neq \chi_{{\rm triv}},
\end{array}\right.}
\end{align*}
where $g(u, v):=\pi^{n}u^p+\alpha v^l+\beta\pi^{n}v^{n+l}$.
One can easily check that $g$ is a polynomial globally non-degenerate
with respect to its Newton polyhedra. Then by Lemma \ref{lem 2.5},
we have
$$Z_g(s, \chi)=\dfrac{\tilde F_1^(q^{-s}, \chi)}
{(1-q^{-1-s})(1-q^{-p-l-pls})},$$
where $\tilde F_1(q^{-s}, \chi)$ is a polynomial with complex coefficients.
Thus
\begin{align}\label{3.5}
Z_f(s, \chi, A_1)={\left\{\begin{array}{rl}
\dfrac{F_1(q^{-s}, \chi)}
{(1-q^{-1-s})(1-q^{-p-l-pls})},
&{\rm if}\ \chi^{n+l}=\chi_{{\rm triv}},\\
0, &{\rm otherwise},
\end{array}\right.}
\end{align}
where $F_1(q^{-s}, \chi)$ is a polynomial with complex coefficients.

For $Z_f(s, \chi, A_2)$, setting $i=\omega,\ j=1$ and $k=0$ in (\ref{eqno 3.1})
and by (\ref{3.2}) and $\omega=\frac{n+l}{p}$, we have
\begin{align*}
&Z_f(s, \chi, A_2)\\
=&q^{-\omega-1}\int_{B_2}
\chi(ac (\pi^{n+l}x^p+\alpha\pi^ny^nz^l+\beta z^{n+l}))
|\pi^{n+l}x^p+\alpha\pi^ny^nz^l+\beta z^{n+l}|^s|dxdydz|,
\end{align*}
where
$$B_2=\{(x, y, z)\in K^3| (\pi^{\omega} x, \pi y, z)\in\mathcal{O}_K^3,
{\rm ord}(\pi^{\omega}x)\geq \omega,\ {\rm ord}(\pi y)\ge 1,\ {\rm ord}(z)=0 \}.$$
Then Lemma \ref{lem 2.9} gives us that $B_2=\mathcal{O}_K^{2}\times\mathcal{O}_K^{\times}$.
Since $\beta\in \mathcal{O}_K^{\times}$, it infers that
$|\pi^{n+l}x^p+\alpha\pi^ny^nz^l+\beta z^{n+l}|=1$ if $(x, y, z)\in B_2$.
Hence one has
\begin{equation}\label{eqno 3.6}
Z_f(s, \chi, A_2)=q^{-\omega-1}\int_{B_2}
\chi(ac (\pi^{n+l}x^p+\alpha\pi^ny^nz^l+\beta z^{n+l}))|dxdydz|
:=C(\chi),
\end{equation}
where $C(\chi)$ is a constant depended on $\chi$ and independent of $s$.

For $Z_f(s, \chi, A_3)$, we set $i=\omega,\ j=k=0$ in (\ref{eqno 3.1}).
Then by (\ref{3.2}), one has
\begin{align*}
&Z_f(s, \chi, A_3)\\
=& q^{-\omega}\int_{B_3}
\chi(ac (\pi^{n+l}x^p+\alpha y^nz^l+\beta z^{n+l}))
|\pi^{n+l}x^p+\alpha y^nz^l+\beta z^{n+l}|^s|dxdydz|,
\end{align*}
where
$$B_3=\{(x, y, z)\in K^3| (\pi^{\omega} x, y, z)\in\mathcal{O}_K^3,
{\rm ord}(\pi^{\omega}x)\geq \omega,\ {\rm ord}(y)={\rm ord}(z)=0 \}.$$
By Lemma \ref{lem 2.9}, we derive that
$B_3=\mathcal{O}_K\times(\mathcal{O}_K^{\times})^2$.

Let $g(x, y, z)=\pi^{n+l}x^p+\alpha y^nz^l+\beta z^{n+l}$. Then
$\bar{g}(x, y, z)=\bar{\alpha} y^nz^l+\bar{\beta} z^{n+l}$.
If $P=(x, y, z)\in {\rm Sing}_{\bar{g}}(\mathbb{F}_q)$, then
$\frac{\partial \bar{g}}{\partial y}=\bar{\alpha}nz^ly^{n-1}=\bar{0}$
implies that $zy=\bar{0}$ since $p\nmid n$. Thus $y=\bar 0$ or $z=\bar 0$.
But the image $\bar{B_3}$ of $B_3=\mathcal{O}_K\times(\mathcal{O}_K^{\times})^2$
under the canonical homomorphism is given by
$\bar{B_3}=\mathbb{F}_q\times(\mathbb{F}_q^{\times})^2$.
Hence we must have
${\rm Sing}_{\bar{g}}(\mathbb{F}_q)\bigcap \bar{B_3}=\emptyset$,
which tells us that $S(g, B_3)=\emptyset$. Using Lemma \ref{lem 2.1},
we obtain that
\begin{equation}\label{eqno 3.7}
Z_f(s, \chi, A_3)=\dfrac{F_3(q^{-s}, \chi)}{1-q^{-1-s}},
\end{equation}
where $F_3(q^{-s}, \chi)$ is a polynomial with complex coefficients.

In what follows, let $a$ be an integer with $0\le a<\omega$.
For any integer $b$ with $4\le b\le 7$, we define the set $A_b^a$ by
$$A_b^a:=\{(x, y, z)\in A_b|{\rm ord}(x)=a\}.$$
Then it follows that $A_b=\bigcup_{0\le a<\omega}A_b^a$ and
\begin{equation}\label{eqno 3.7'}
Z_f(s, \chi, A_b)=\sum_{a=0}^{\omega-1}{Z_f(s, \chi, A_b^a)}.
\end{equation}
For $Z_f(s, \chi, A_4^a)$, we set $i=a,\ j=k=1$ in (\ref{eqno 3.1}).
Therefore by (\ref{3.2}) and (\ref{eqno 3.7'}), we get that
\begin{align*}
Z_f(s, \chi, A_4)=&\sum_{a=0}^{\omega-1}{Z_f(s, \chi, A_4^a)}\\
=&\sum_{a=0}^{\omega-1}q^{-a-2}\int_{B_4^a}
\chi(ac (\pi^{ap}x^p+\alpha\pi^{n+l}y^nz^l+\beta\pi^{n+l}z^{n+l}))\\
&\times|\pi^{ap}x^p+\alpha\pi^{n+l}y^nz^l+\beta\pi^{n+l}z^{n+l}|^s|dxdydz|.\\
=&\sum_{a=0}^{\omega-1}q^{-a-2-aps}\int_{B_4^a}
\chi(ac (x^p+\alpha\pi^{n+l-ap}y^nz^l+\beta\pi^{n+l-ap}z^{n+l}))\\
&\times|x^p+\alpha\pi^{n+l-ap}y^nz^l+\beta\pi^{n+l-ap}z^{n+l}|^s|dxdydz|,
\end{align*}
where
$$B_4^a=\{(x, y, z)\in K^3| (\pi^{a} x, \pi y, \pi z)\in\mathcal{O}_K^3,
{\rm ord}(\pi^{a}x)=a,\ {\rm ord}(y)\ge 1,\ {\rm ord}(z)\ge 1 \}.$$
So Lemma \ref{lem 2.9} tells us that
$B_4^a=\mathcal{O}_K^{\times}\times\mathcal{O}_K^2$. Since
$x\in \mathcal{O}_K^{\times}$, we have
$|x^p+\alpha\pi^{n+l-ap}y^nz^l+\beta\pi^{n+l-ap}z^{n+l}|=1$
for any $0\leq a<\omega$ and $(x, y, z)\in B_4^a$.
Therefore we arrive at
\begin{align}\label{3.8}
Z_f(s, \chi, A_4)
=&\sum_{a=0}^{\omega-1}q^{-a-2-aps}\int_{B_4^a}
\chi(ac (x^p+\alpha\pi^{n+l-ap}y^nz^l
+\beta\pi^{n+l-ap}z^{n+l}))|dxdydz|\nonumber\\
:=&F_4(q^{-s}, \chi),
\end{align}
where $F_4(q^{-s}, \chi)$ is a polynomial with complex coefficients.

For $Z_f(s, \chi, A_5^a)$, we set $i=a,\ j=0,\ k=1$ in (\ref{eqno 3.1}).
Then by (\ref{3.2}) and (\ref{eqno 3.7'}), one has
\begin{align}\label{3.9}
&Z_f(s, \chi, A_5)=\sum_{a=0}^{\omega-1}{Z_f(s, \chi, A_5^a)}
=\sum_{a=0}^{\omega-1}q^{-a-1}Z_{f_{5, a}}(s, \chi, B_5^a)
\end{align}
where $f_{5, a}:=\pi^{ap}x^p+\alpha\pi^{l}y^nz^l+\beta\pi^{n+l}z^{n+l}$
and
$$B_5^a=\{(x, y, z)\in K^3| (\pi^{a} x, y, \pi z)\in\mathcal{O}_K^3,
{\rm ord}(\pi^{a}x)=a,\ {\rm ord}(y)=0, {\rm ord}(z)\ge 1 \}.$$
But Lemma \ref{lem 2.9} gives us that
$B_5^a=(\mathcal{O}_K^{\times})^2\times\mathcal{O}_K$.
Since $\alpha\in\mathcal{O}_K^{\times}$, $\beta^{1/p}\in\mathcal{O}_K^{\times}$,
$p\mid (n+l)$ and $p\nmid l$, it then follows from Lemma \ref{lem 2.8} that
\begin{align}\label{3.10}
Z_{f_{5, a}}(s, \chi, B_5^a)=&{\left\{\begin{array}{rl}
&F_{5, a}(q^{-s}, \chi),\ \ \ \ \ \ {\rm if} \ 0\leq a<\dfrac{l}{p},\\
&\dfrac{\tilde F_{5, a}(q^{-s}, \chi)}{1-q^{-1-s}},\ \ \ \ \
{\rm if} \ \dfrac{l}{p}\leq a<\omega,
\end{array}\right.}
\end{align}
where $F_{5, a}(q^{-s}, \chi)$ and $\tilde F_{5, a}(q^{-s}, \chi)$
are polynomials with complex coefficients.
By (\ref{3.9}) and (\ref{3.10}), one derives that
\begin{equation}\label{eqno 3.11}
Z_f(s, \chi, A_5)=\dfrac{F_5(q^{-s}, \chi)}{1-q^{-1-s}},
\end{equation}
where $F_5(q^{-s}, \chi)$ is a polynomial with complex coefficients.

For $Z_f(s, \chi, A_6^a)$, we set $i=a,\ j=1,\ k=0$ in (\ref{eqno 3.1}).
By (\ref{3.2}) and (\ref{eqno 3.7'}), one deduces that
\begin{align}\label{3.12}
Z_f(s, \chi, A_6)
&=\sum_{a=0}^{\omega-1}{Z_f(s, \chi, A_6^a)}\nonumber\\
&=\sum_{a=0}^{\omega-1}q^{-1-a}Z_{f_{6, a}}(s, \chi, B_6^a),
\end{align}
where $f_{6, a}(x, y, z):=\pi^{ap}x^p+\alpha\pi^ny^nz^l+\beta z^{n+l}$
and
$$B_6^a=\{(x, y, z)\in K^3| (\pi^{a} x, \pi y, z)\in\mathcal{O}_K^3,
{\rm ord}(\pi^{a}x)=a,\ {\rm ord}(y)\ge 1, {\rm ord}(z)=0 \}.$$
By Lemma \ref{lem 2.9}, we have
$B_6^a=\mathcal{O}_K^{\times}\times\mathcal{O}_K\times\mathcal{O}_K^{\times}$.

For any positive integer $a$ with $a<\omega$, since $\beta \in\mathcal{O}_K^{\times}$,
one derives that
$|\pi^{ap}x^p+\alpha\pi^ny^nz^l+\beta z^{n+l}|=1$ if $(x, y, z)\in B_6^a$.
Then it follows that
\begin{equation}\label{eqno 3.13}
Z_{f_{6, a}}(s, \chi, B_6^a)
=\int_{B_6^a}\chi(ac (\pi^{ap}x^p+\alpha\pi^ny^nz^l+\beta z^{n+l}))|dxdydz|
:=c(a, \chi),
\end{equation}
where $c(a, \chi)$ is a constant depended on $a$ and $\chi$ and
independent of $s$. For the case $a=0$, since
$\alpha\in\mathcal{O}_K^{\times}$, $\beta^{1/p}\in\mathcal{O}_K^{\times}$,
$p\mid (n+l)$ and $p\nmid l$, Lemma \ref{lem 2.7} tells us that
\begin{equation}\label{eqno 3.14}
Z_{f_{6, 0}}(s, \chi, B_6^a)=\dfrac{\tilde F_6(q^{-s}, \chi)}
{(1-q^{-1-s})(1-q^{-p-n-pns})},
\end{equation}
where $\tilde F_6(q^{-s}, \chi)$ is a polynomial with complex coefficients.
Putting (\ref{eqno 3.13}) and (\ref{eqno 3.14}) into (\ref{3.12}),
we arrive at
\begin{equation}\label{eqno 3.15}
Z_f(s, \chi, A_6)=\dfrac{F_6(q^{-s}, \chi)}{(1-q^{-1-s})(1-q^{-p-n-pns})},
\end{equation}
where $F_6(q^{-s}, \chi)$ is a polynomial with complex coefficients.

For $Z_f(s, \chi, A_7^a)$, we set $i=a,\ j=k=0$ in (\ref{eqno 3.1}).
Then it follows from (\ref{3.2}) and (\ref{eqno 3.7'}) that
\begin{align}\label{3.16}
Z_f(s, \chi, A_7)&=\sum_{a=0}^{\omega-1}{Z_f(s, \chi, A_7^a)}\nonumber\\
&=\sum_{a=0}^{\omega-1}q^{-a}{Z_{f_{7, a}}(s, \chi, B_7^a)},
\end{align}
where $f_{7, a}(x, y, z):=\pi^{ap}x^p+\alpha y^nz^l+\beta z^{n+l}$ and
$$B_7^a=\{(x, y, z)\in K^3| (\pi^{a} x, y, z)\in\mathcal{O}_K^3,
{\rm ord}(\pi^{a}x)=a,\ {\rm ord}(y)={\rm ord}(z)=0 \}.$$
Notice that Lemma \ref{lem 2.9} tells us that $B_7^a=(\mathcal{O}_K^{\times})^3$.

For any integer $a$ with $0\leq a<\omega$, if
$P=(x, y, z)\in {\rm Sing}_{\bar{f}_{7, a}}(\mathbb{F}_q)$,
then $\frac{\partial \bar{f}_{7, a}}{\partial y}=\bar{\alpha}ny^{n-1}z^l=\bar{0}$
implies that $yz=\bar{0}$ since $p\nmid n$. Since
$B_7^a=(\mathcal{O}_K^{\times})^3$, we have $\bar{B_7^a}=(\mathbb{F}_q^{\times})^3$.
Hence we have
${\rm Sing}_{\bar{f}_{7, a}}(\mathbb{F}_q)\bigcap \bar{B_7^a}=\emptyset$,
which implies that $S(f_{7, a}, B_7^a)=\emptyset$. Then by Lemma \ref{lem 2.1},
one has
\begin{equation}\label{eqno 3.17}
Z_{f_{7, a}}(s, \chi, B_7^a)=\dfrac{F_{7, a}(q^{-s}, \chi)}{1-q^{-1-s}},
\end{equation}
where $F_{7, a}(q^{-s}, \chi)$ is a polynomial with complex coefficients.
Then by (\ref{eqno 3.17}) and (\ref{3.16}), we obtain that
\begin{equation}\label{eqno 3.18}
Z_f(s, \chi, A_7)=\dfrac{F_7(q^{-s}, \chi)}{1-q^{-1-s}},
\end{equation}
where $F_7(q^{-s}, \chi)$ is a polynomial with complex coefficients.

Finally, combining (\ref{eqno 3.4})-(\ref{eqno 3.7}), (\ref{3.8}),
(\ref{eqno 3.11}), (\ref{eqno 3.15}) with (\ref{eqno 3.18}) gives
us the desired result. This finishes the proof of Theorem
\ref{thm 1.1}. \hfill$\Box$\\
\\
{\bf Remark 3.1.}
When we calculate $Z_f(s, \chi, B_1)$, we have used Lemma \ref{lem 2.5}
to calculate $Z_g(s, \chi, B_1)$,
where $g(u, v)=\pi^{n}u^p+\alpha v^l+\beta\pi^{n}v^{n+l}$ and $l>2$.
But if $l=1$, we cannot use this lemma because $g$ is not a polynomial
globally non-degenerate with respect to its Newton polyhedra
(since the origin of $K^n$ is not a singular point of $g$). Instead,
let $g_1(x, y, z)=\pi^{n}x^p+\alpha zy^n+\beta\pi^{n}z^{n+1}$,
one can easily derive that $S(g_1, D_1)=\emptyset$.
Thus by Lemma \ref{lem 2.1}, we get that
$$Z_f(s, \chi, B_1)=\dfrac{M(q^{-s}, \chi)}{1-q^{-1-s}},$$
where $M(q^{-s}, \chi)$ is a polynomial with complex coefficients.
This explains why in our situation that $l>1$, we can obtain a new extra
candidate pole which does not appear in the case $l=1$ due to \cite{[LIS]}.

\begin{exm}
Let $K$ be a non-archimedean local field with
characteristic $3$, and denote its residue field by $\mathbb{F}_q$,
where $q=3^r$ and $r$ is an odd integer. Take $\chi=\chi_{{\rm triv}}$.
We consider the hybrid polynomial of following form:
\begin{align}
g(x, y, z)=x^3+yz^2\sum_{i=0}^{3}{\dbinom{4}{1+i}y^i(z-y)^{3-i}}.
\end{align}
Then by (\ref{eqno 1.4}), $f(x, y, z)=x^3+y^4z^2+z^6$.
An explicit calculation of the seven integrals of (\ref{eqno 3.4})
shows that
\begin{align*}
Z_f(s, \chi, B_1)
&=\dfrac{q^{-3}t^2(1-q^{-1})^2}{(1-q^{-5}t^6)(1-q^{-1}t)}
\Big(-q^{-6}t^8+q^{-6}t^7+(q^{-4}-q^{-5})t^6-q^{-4}t^5\\
&+q^{-2}t^4-q^{-2}t^3+q^{-1}t^2-q^{-1}t+1\Big),\\
Z_f(s, \chi, B_2)&=q^{-3}(1-q^{-1}),\\
Z_f(s, \chi, B_3)&=q^{-2}(1-q^{-1})^2,\\
Z_f(s, \chi, B_4)&=(1-q^{-1})(q^{-2}+q^{-3}t^3),\\
Z_f(s, \chi, B_5)
&=(1-q^{-1})^2(q^{-3}t^3+(q^{-2}-q^{-3})t^2+q^{-1}),\\
Z_f(s, \chi, B_6)
&=\dfrac{q^{-2}(1-q^{-1})^2}{(1-q^{-1}t)(1-q^{-7}t^{12})}
\Big(-(q^{-7}-q^{-8})t^{13}+(q^{-5}-q^{-7})t^{12}-q^{-5}t^{10}\\
&+q^{-3}t^8-q^{-3}t^7+q^{-2}t^6-q^{-2}t^5+t^3-q^{-1}t+1\Big)\\
Z_f(s, \chi, B_7)
&=\dfrac{(1-q^{-1})^2}{1-q^{-1}t}(q^{-3}t+1-q^{-1}-q^{-2}),
\end{align*}
where $t:=q^{-s}$. Then by (\ref{eqno 3.4}),
we deduce that all the candidate poles in
Theorem \ref{thm 1.1} are indeed the poles of $Z_f(s, \chi)$.
\end{exm}

\end{document}